\RequirePackage{fix-cm}
\documentclass{amsart}
\usepackage[francais,english]{babel}
\usepackage[T1]{fontenc}
\usepackage[latin1]{inputenc}
\usepackage{amsmath,amssymb,mathrsfs} 
\usepackage[all,2cell]{xy}
\usepackage{graphicx}
\usepackage{enumerate} 
\usepackage{mathrsfs} 
\usepackage{dsfont}
\usepackage{lmodern}
\DeclareMathOperator{\ch}{ch}
\DeclareMathOperator{\td}{td}
\DeclareMathOperator{\eu}{eu}

\DeclareMathOperator{\RHom}{RHom}

\DeclareMathOperator{\fRHom}{R\mathcal{H}om}

\DeclareMathOperator{\Hom}{Hom}

\DeclareMathOperator{\HH}{HH}
\DeclareMathOperator{\bHH}{\mathds{HH}}

\DeclareMathOperator{\hh}{hh}
\DeclareMathOperator{\id}{id}
\DeclareMathOperator{\opp}{op}
\DeclareMathOperator{\Mod}{Mod}

\DeclareMathOperator{\Hn}{H}
\DeclareMathOperator{\tr}{Tr}

\DeclareMathOperator{\gr}{gr}
\DeclareMathOperator{\Rg}{R\Gamma}
\DeclareMathOperator{\dR}{R}
\DeclareMathOperator{\Ob}{Ob}
\DeclareMathOperator{\Supp}{Supp}
\DeclareMathOperator{\sGamma}{\Gamma}

\begin{document}



\newcommand{\On}[1]{\mathcal{O}_{#1}}
\newcommand{\En}[1]{\mathcal{E}_{#1}}
\newcommand{\Fn}[1]{\mathcal{F}_{#1}} 
\newcommand{\tFn}[1]{\mathcal{\tilde{F}}_{#1}}
\newcommand{\hum}[1]{hom_{\mathcal{A}}({#1})}
\newcommand{\hcl}[2]{#1_0 \lbrack #1_1|#1_2|\ldots|#1_{#2} \rbrack}
\newcommand{\hclp}[3]{#1_0 \lbrack #1_1|#1_2|\ldots|#3|\ldots|#1_{#2} \rbrack}
\newcommand{\catMod}{\mathsf{Mod}}
\newcommand{\Der}{\mathsf{D}}
\newcommand{\Ds}{D_{\mathbb{C}}}
\newcommand{\DG}{\mathsf{D}^{b}_{dg,\mathbb{R}-\mathsf{C}}(\mathbb{C}_X)}
\newcommand{\lI}{[\mspace{-1.5 mu} [}
\newcommand{\rI}{] \mspace{-1.5 mu} ]}
\newcommand{\Ku}[2]{\mathfrak{K}_{#1,#2}}
\newcommand{\iKu}[2]{\mathfrak{K^{-1}}_{#1,#2}}
\newcommand{\Be}{B^{e}}
\newcommand{\op}[1]{#1^{\opp}}
\newcommand{\N}{\mathbb{N}}
\newcommand{\Ab}[1]{#1/\lbrack #1 , #1 \rbrack}
\newcommand{\Du}{\mathbb{D}}
\newcommand{\C}{\mathbb{C}}
\newcommand{\Z}{\mathbb{Z}}
\newcommand{\w}{\omega}
\newcommand{\K}{\mathcal{K}}
\newcommand{\Hoc}{\mathcal{H}\mathcal{H}}
\newcommand{\env}[1]{{\vphantom{#1}}^{e}{#1}}
\newcommand{\eA}{{}^eA}
\newcommand{\eB}{{}^eB}
\newcommand{\eC}{{}^eC}
\newcommand{\cA}{\mathcal{A}} 
\newcommand{\cB}{\mathcal{B}}
\newcommand{\cR}{\mathcal{R}}
\newcommand{\cL}{\mathcal{L}}
\newcommand{\cO}{\mathcal{O}}
\newcommand{\cM}{\mathcal{M}}
\newcommand{\cN}{\mathcal{N}}
\newcommand{\cK}{\mathcal{K}}
\newcommand{\cC}{\mathcal{C}}
\newcommand{\Hper}{\Hn^0_{\textrm{per}}}
\newcommand{\Dper}{\Der_{\mathrm{perf}}}
\newcommand{\Yo}{\textrm{Y}}
\newcommand{\gqcoh}{\mathrm{gqcoh}}
\newcommand{\coh}{\mathrm{coh}}
\newcommand{\cc}{\mathrm{cc}}
\newcommand{\qcc}{\mathrm{qcc}}
\newcommand{\qcoh}{\mathrm{qcoh}}
\newcommand{\obplus}[1][i \in I]{\underset{#1}{\overline{\bigoplus}}}
\newcommand{\Lte}{\mathop{\otimes}\limits^{\rm L}}
\newcommand{\pt}{\textnormal{pt}}
\newcommand{\A}[1][X]{\cA_{{#1}}}
\newcommand{\dA}[1][X]{\cC_{X_{#1}}}
\newcommand{\conv}[1][]{\mathop{\circ}\limits_{#1}}
\newcommand{\sconv}[1][]{\mathop{\ast}\limits_{#1}}
\newcommand{\reim}[1]{\textnormal{R}{#1}_!}
\newcommand{\roim}[1]{\textnormal{R}{#1}_\ast}
\newcommand{\ldetens}{\overset{\mathrm{L}}{\underline{\boxtimes}}}
\newcommand{\br}{\bigr)}
\newcommand{\bl}{\bigl(}
\newcommand{\sC}{\mathscr{C}}
\newcommand{\ucat}{\mathbf{1}}
\newcommand{\ubtimes}{\underline{\boxtimes}}
\newcommand{\uLte}{\mathop{\underline{\otimes}}\limits^{\rm L}} 
\newcommand{\Lp}{\mathrm{L}p}

\newtheorem{theorem}{Theorem}[section]
\newtheorem{lemma}[theorem]{Lemma}
\newtheorem{Notation}{Notation}
\newtheorem{proposition}{Proposition}
\newtheorem{corollary}{Corollary}

\theoremstyle{definition}
\newtheorem{definition}[theorem]{Definition}
\newtheorem{example}[theorem]{Example}

\theoremstyle{remark}
\newtheorem{remark}[theorem]{Remark}

\numberwithin{equation}{section}

\title[The Lefschetz-Lunts formula for DQ-modules]{The Lefschetz-Lunts formula for deformation quantization modules}



\author{Fran\c{c}ois Petit
}


\address{F. Petit, Max Planck Institute for Mathematics, Vivatsgasse 7, 53111 Bonn, Germany.}
\email{petit@mpim-bonn.mpg.de}           


\begin{abstract}
We adapt to the case of deformation quantization modules a formula of V.~Lunts~\cite{Lunts} who 
calculates the trace of a kernel acting on Hochschild homology.
\end{abstract}

\maketitle

\section{Introduction}

Inspired by the work of D. Shklyarov (see \cite{shklyarov}), V. Lunts has established in \cite{Lunts} a Lefschetz type formula which calculates the trace of a coherent kernel acting on the Hochschild homology of a projective variety (Theorem \ref{lunts1}). This result has inspired several other works (\cite{cisinski,poly}). In \cite{cisinski}, Cisinski and Tabuada recover the result of Lunts via the theory of non-commutative motives. In \cite{poly}, Polischuk proves similar formulas and applies them to matrix factorisation. The aim of this paper is to adapt Lunts formula to the case of deformation quantization modules (DQ-modules) of Kashiwara-Schapira on complex Poisson manifolds. For that purpose, we develop an abstract framework which allows one to obtain Lefschetz-Lunts type formulas in symmetric monoidal categories endowed with some additional data. 

 Our proof relies essentially on two facts. The first one is that the composition operation on the Hochschild homology is compatible in some sense with the symmetric monoidal structures of the categories involved. The second one is the functoriality of the Hochschild class with respect to composition of kernels. This suggest that the Lefeschtz-Lunts formula is a 2-categorical statement and that it might be possible to build a set-up, in the spirit of \cite{caldararu}, which would encompass simultaneously these two aspects. 

Let us compare briefly the different approaches and settings of \cite{Lunts}, \cite{cisinski} and \cite{poly} to ours. As already mentioned, we are working in the framework of deformation quantization modules over complex manifolds. 

The approach of Lunts is based on a certain list of properties of the Hochschild homology of algebraic varieties (see \cite[§3]{Lunts}). These properties mainly concern the behaviour of Hochschild homology with respect to the composition of kernels and its functoriality. A straightforward consequence of these properties is that the morphism $X \to \pt$ induces a map from the Hochschild homology of $X$ to the ground field $k$. Such a map does not exist in the theory of DQ-modules. Thus, it is not possible to integrate a single class with values in Hochschild homology and one has to integrate a pair of classes. Then, it seems that the method of V. Lunts cannot be carried out in our context.

In \cite{cisinski}, the authors showed that the results of V. Lunts for projective varieties can be derived from a very general statement for additive invariants of smooth and proper differential graded category in the sense of Kontsevich. However, it is not clear that this approach would work for DQ-modules even in the algebraic case. Indeed, the  results used to relate non-commutative motives to more classical geometric objects rely on the existence of a compact generator for the derived category of quasi-coherent sheaves which is a classical generator of the derived category of coherent sheaves. To the best of our knowledge, there are no such results for DQ-modules. Similarly, the approach of \cite{poly} does not seem to be applicable to DQ-modules.

The paper is organised as follow. In the first part, we sketch a formal framework in which we can get a formula for the trace of a class acting on a certain homology, starting from a symmetric monoidal category endowed with some specific data. In the second part, we briefly review, following \cite{KS3}, some elements of the theory of DQ-modules. The last part is mainly devoted to the proof of the Lefschetz-Lunts theorems for DQ-modules. Then, we briefly explain how to recover some of Lunts's results.

\begin{flushleft}
\textbf{Acknowledgement:} I would like to thank Damien Calaque and Michel Vaquié for their careful reading of the manuscript and numerous suggestions which have allowed substantial improvements.
\end{flushleft}

\section{A general framework for Lefschetz type theorems} \label{generalframe}

\subsection{A few facts about symmetric monoidal categories and traces}
In this subsection, we recall a few classical facts concerning dual pairs and traces in symmetric monoidal categories. References for this subsection are \cite[Chap.4]{categories_and_sheaves}, \cite{may}, \cite{ponto}.

Let $\mathscr{C}$ be a symmetric monoidal category with product $\otimes$, unit object $\mathbf{1}_\sC$ and symmetry isomorphism $\sigma$. All along this paper, we identify $(X \otimes Y) \otimes Z$ and $X \otimes (Y \otimes Z)$.

\begin{definition}
We say that $X \in \Ob(\sC)$ is dualizable if there is $Y \in \Ob(\sC)$ and two morphisms, $\eta: \mathbf{1}_\sC \to X \otimes Y$,  $\varepsilon: Y \otimes X \to \mathbf{1}_\sC$ called coevaluation and evaluation such that the condition (a) and (b) are satisfied:
\begin{enumerate}[(a)]
\item The composition $X\simeq \ucat_\sC \otimes X \stackrel{\eta \otimes \id_X}{\to} X \otimes Y \otimes X \stackrel{\id_X \otimes \varepsilon}{\to} X \otimes \ucat_\sC \simeq X$ is the identity of $X$.

\item The composition $Y\simeq Y \otimes \ucat_\sC \stackrel{\id_Y \otimes \eta}{\to} Y \otimes X \otimes Y \stackrel{\varepsilon \otimes \id_Y}{\to} \ucat_\sC \otimes Y \simeq Y$ is the identity of $Y$.
\end{enumerate}
We call $Y$ a dual of $X$ and say that $(X,Y)$ is a dual pair.
\end{definition}

We shall prove that some diagrams commute. For that purpose recall the useful lemma below communicated to us by Masaki Kashiwara.

\begin{lemma}\label{kas}
Let $\mathscr{C}$ be a monoidal category with unit. Let $(X, Y)$ be a dual pair with coevaluation and evaluation morphisms  
\begin{equation*}
\ucat_\sC \stackrel{\eta}{\to}X \otimes Y, \; Y \otimes X\stackrel{\varepsilon}{\to}\ucat_\sC.
\end{equation*}
Let $f:\ucat_\sC \to X \otimes Y$ be a morphism such that $(\id_X \otimes \varepsilon) \circ (f \otimes \id_X)= \id_X$. Then $f= \eta$.
\end{lemma}
\begin{proof}
Consider the diagram
\begin{equation*}
\xymatrix @C=3pc{
\ucat_\sC \ar[r]^-{\eta}\ar[d]_-{f}&X\otimes Y\ar[d]^-{f \otimes \id_X\otimes \id_Y }&\\
X \otimes Y \ar[r]_-{\id_X \otimes \id_Y \otimes \eta}& X\otimes Y \otimes X \otimes Y \ar[rd]|-{\id_X \otimes\varepsilon\otimes \id_Y}&\\
&& X\otimes Y.
}
\end{equation*}
By the hypothesis, $(\id_X\otimes\varepsilon\otimes \id_Y)\circ (f \otimes \id_X\otimes \id_Y)=\id_X \otimes \id_Y$ and 
$(\id_X \otimes \varepsilon \otimes \id_Y)\circ (\id_X\otimes \id_Y \otimes \eta)=(\id_X \otimes \id_Y)$. Therefore, $\eta=f$. 
\end{proof}

The next proposition is well known. But, we do not the original reference. A proof can be found in \cite[Chap.4]{categories_and_sheaves}.

\begin{proposition}
If (X,Y) is a dual pair, then for every $Z, \; W \in \Ob(\sC)$, there are natural isomorphisms
\begin{align*}
\Phi:\Hom_\sC(Z,W \otimes Y) \stackrel{\sim}{\to} \Hom_\sC(Z \otimes X, W),\\
\Psi:\Hom_\sC(Y \otimes Z , W) \stackrel{\sim}{\to} \Hom_\sC( Z , X \otimes W )
\end{align*}
 where for $f \in \Hom_\sC(Z,W \otimes Y)$ and $g \in \Hom_\sC(Y \otimes Z , W)$, 
\begin{align*}
\Phi(f)= (\id_W \otimes \varepsilon) \circ (f \otimes \id_X),\\  
\Psi(g)=(\id_X \otimes g) \circ ( \eta \otimes \id_Z).
\end{align*}
\end{proposition}

\begin{remark}\label{independance}
It follows that $Y$ is a representative of the functor $Z \mapsto \Hom_\sC(Z \otimes X, \ucat_\sC)$ as well as a representative of the functor
$W \mapsto \Hom_\sC(\ucat_\sC, X \otimes W)$. Therefore, the dual of a dualizable object is unique up to a unique isomorphism.
\end{remark}

\begin{definition}\label{trace}
For a dualizable object $X$, the trace of $f:X \to X$ denoted $\tr(f)$ is the composition
\begin{equation*}
\ucat_\sC \to X \otimes Y \stackrel{f \otimes \id}{\to} X \otimes Y \stackrel{\sigma}{\to} Y \otimes X \stackrel{\varepsilon}{\to} \ucat_\sC.
\end{equation*} 
Then, $\tr(f) \in \Hom_{\sC}(\ucat_\sC,\ucat_\sC)$.
\end{definition}

\begin{remark}
The trace could also by defined as the following composition
\begin{equation*}
\ucat_\sC \to X \otimes Y \stackrel{\sigma}{\to}  Y \otimes X  \stackrel{\id \otimes f}{\to}Y \otimes X \stackrel{\varepsilon}{\to} \ucat_\sC.
\end{equation*} 
These two definitions of the trace coincide because $(\id \otimes f)\sigma= \sigma (f \otimes \id) $ since $\sigma$ is a natural transformation.
\end{remark}
Recall the following fact. 
\begin{lemma}\label{invar}
With the notation of Definition \ref{trace}, the trace is independent of the choice of a dual for X.
\end{lemma}

\begin{proof}
Let $Y$ and $Y^\prime$ two duals of $X$ with evaluations $\varepsilon$, $\varepsilon^\prime$ and coevalution $\eta$ and $\eta^\prime$. By definition of a representative of the functor $Z \mapsto \Hom_\sC(Z \otimes X, \ucat_\sC)$ there exist a unique isomorphism $\theta: Y \to Y^\prime$ such that the diagram 
\begin{equation*}
\xymatrix{\Hom_\sC(Z,Y^\prime) \ar[r]^-{\Phi^\prime}& \Hom_\sC(Z \otimes X,\ucat_\sC)\\
\Hom_\sC(Z,Y) \ar[u]^{\theta \circ} \ar[ru]_-{\Phi}&.
} 
\end{equation*}
commutes. For $Z=Y$, the diagram, applied to $\id_Y$, implies $\varepsilon=\varepsilon^\prime \circ (\theta \otimes \id_X)$. Using Lemma \ref{kas}, we get that $\eta= (\id_X \otimes \theta^{-1}) \circ \eta^\prime$. It follows that the diagram

\begin{equation*}
\xymatrix{& X \otimes Y \ar[r]^-{f \otimes \id} \ar[dd]^-{\id  \otimes \theta} & X \otimes Y \ar[r]^-{\sigma} \ar[dd]^-{\id  \otimes \theta} & Y \otimes X \ar[dd]^{\theta \otimes \id} \ar[rd]^-{\varepsilon}&\\
\ucat_\sC \ar[ru]^-{\eta} \ar[rd]_-{\eta^\prime}&&&& \ucat_\sC \\
& X \otimes Y^\prime \ar[r]_-{f \otimes \id} & X \otimes Y^\prime \ar[r]_-{\sigma} &  Y^\prime \otimes X \ar[ru]_-{\varepsilon^\prime}&
}
\end{equation*}
commutes which proves the claim.
\end{proof} 

\begin{example}\label{example}(see \cite[Chap.3]{may})
Let $k$ be a Noetherian commutative ring of finite weak global dimension. Let $\Der^b(k)$ be the bounded derived category of the category of $k$-modules. It is a symmetric monoidal category for $\Lte_k$. We denote by $\Der^b_f(k)$, the full subcategory of $\Der^b(k)$ whose objects are the complexes with finite type cohomology. If $M \in \Ob(\Der^b_f(k))$, its dual is given by $\RHom_k(M,k)$. The evaluation and the coevaluation are given by
\begin{align*}
\mathrm{ev}: \RHom_k(M,k)& \Lte_k M \to k\\
\mathrm{coev}: k \to \RHom_k(M,M) \stackrel{\sim}{\leftarrow} & M \Lte_k \RHom_k(M,k).
\end{align*}
If we further assume that $k$ is an integral domain, then $k$ can be embedded into its field of fraction $\mathrm{F}(k)$. If $f$ is an endomorphism of $M$ then the trace of $f$
\begin{equation*}
k \stackrel{\mathrm{coev}}{\to} M \otimes \RHom_k(M,k) \stackrel{f \otimes \id}{\to} M \otimes \RHom_k(M,k) \stackrel{\sigma}{\to} \RHom_k(M,k) \otimes M \stackrel{\mathrm{ev}}{\to} k
\end{equation*}
coincides with $\sum_i (-1)^i \tr(\Hn^i(\id_{\mathrm{F}(k)} \otimes f))$. If $f=\id_M$, one sets 
\begin{equation*}
\chi(M)=\sum_{i \in \Z}(-1)^i \dim_{\mathrm{F}(k)}(\Hn^i(M)).
\end{equation*}\index{chi@$\chi$}
\end{example}

\subsection{The framework}\label{monoidalframework}
In this section, we define a general framework for Lefschetz-Lunts type theorems. Let $\mathscr{C}$ be a symmetric monoidal category with product $\otimes$, unit object $\mathbf{1}_{\mathscr{C}}$ and symmetry isomorphism $\sigma$. Let $k$ be a Noetherian commutative ring with finite cohomological dimension.

Assume we are given: 
\begin{enumerate}[(a)]
\item a monoidal functor $(\cdot)^a:\mathscr{C} \to \mathscr{C}$ such that $(\cdot)^a \circ (\cdot)^a = \id_\mathscr{C}$ and $\mathbf{1}_{\mathscr{C}}^a \simeq \mathbf{1}_{\mathscr{C}}$

\item a symmetric monoidal functor $(L,\mathfrak{K}): \mathscr{C} \to \Der^b(k)$ where $\mathfrak{K}$ is the isomorphism of bifunctor from  $L(\cdot) \Lte L(\cdot)$ to $L( \cdot \otimes \cdot)$. That is
$L(X) \Lte L(Y) \stackrel{\mathfrak{K}}{\simeq} L(X \otimes Y)$ naturally in $X$ and $Y$ and $L(\ucat_\sC)\simeq k$,

\item for $X_i \in \Ob(\mathscr{C})$ $(i=1, 2, 3)$, a morphism
\begin{equation*}
\underset{2}{\cup}: L(X_1 \otimes X_2^a) \Lte L(X_2 \otimes X_3^a) \to L(X_1 \otimes X_3^a),
\end{equation*}

\item for every $X \in \Ob(\mathscr{C})$, a morphism
\begin{equation*}
L_{\Delta_X}:k \to L(X \otimes X^a),
\end{equation*}
\end{enumerate}
these data verifying the following properties:

\begin{enumerate}[(P1)]
\item \label{compku} for $X_1, \; X_3 \in \Ob(\sC)$, the diagram
\begin{equation*}
\xymatrix{ L(X_1 \otimes \mathbf{1}_{\mathscr{C}}^a ) \Lte L(\mathbf{1}_{\mathscr{C}} \otimes X_3) \ar[r]^-{\underset{\mathbf{1}_{\mathscr{C}}}{\cup}} \ar[d]^-{\wr}&  L(X_1 \otimes X_3)\ar[d]^{\id}\\
L(X_1) \Lte L(X_3) \ar[r]^{\mathfrak{K}}& L(X_1 \otimes X_3)
}
\end{equation*}  
commutes,

\item \label{assol} for $X_1, \; X_2, \; X_3, \; X_4 \in \Ob(\mathscr{C})$, the diagram 
\begin{equation*}
\scalebox{0.93}{\xymatrix{
L(X_1 \otimes X_2^a) \Lte L(X_2 \otimes X_3^a) \Lte L(X_3 \otimes X_4^a) \ar[r]^-{\underset{2}{\cup} \otimes \id} \ar[d]_-{\id \otimes \underset{3}{\cup}}& L(X_1 \otimes X_3^a) \Lte L(X_3 \otimes X_4) \ar[d]^-{\underset{3}{\cup}}\\
L(X_1 \otimes X_2^a) \Lte L(X_2 \otimes X_4^a) \ar[r]^-{\underset{2}{\cup}}& L(X_1 \otimes X_4^a)
}}
\end{equation*}
commutes,

\item \label{viceversa} the diagram 

\begin{equation*}
\xymatrix{
k \ar[r]^-{L_{\Delta_X}} \ar[rd]_-{L_{\Delta_{X^a}}} & L(X \otimes X^a) \ar[d]^-{L(\sigma)} \\
& L(X^a \otimes X)
}
\end{equation*}
commutes,

\item \label{idcomp} the composition
\begin{equation*}
L(X) \stackrel{L_{\Delta_X} \otimes \id_{L(X)}}{\longrightarrow}  L(X \otimes X^a) \Lte L(X) \stackrel{\underset{X}{\cup}}{\to} L(X)
\end{equation*}
is the identity of $L(X)$ and the composition
\begin{equation*}
L(X^a) \stackrel{\id_{L(X^a)} \otimes L_{\Delta_X}}{\longrightarrow} L(X^a) \Lte L(X \otimes X^a) \stackrel{\underset{X^a}{\cup}}{\to} L(X^a) 
\end{equation*} 
is the identity of $L(X^a)$,

\item \label{pairingdiag} the diagram
\begin{equation*}
\xymatrix{ L(X \otimes X^a) \Lte L(X^a \otimes X )\ar[r]^-{\underset{X^a \otimes X}{\cup}}& k \\
 L(X^a) \Lte L(X) \ar[u]^{L_{\Delta_X} \otimes \mathfrak{K}} \ar[ru]_-{\underset{X}{\cup}}.
} 
\end{equation*}
commutes,

\item \label{switch} for $X_1$ and $X_2$ belonging to $\Ob(\sC)$, the diagram
\begin{equation*}
\xymatrix{ L((X_1 \otimes X_2)^a) \Lte L((X_1 \otimes X_2)) \ar[r]^-{\underset{X_1 \otimes X_2}{\cup}}& k \\
L((X_2 \otimes X_1)^a) \Lte L(X_2 \otimes X_1) \ar[ru]
_-{\underset{X_2 \otimes X_1}{\cup}} \ar[u]^-{L(\sigma) \otimes L(\sigma)}
}
\end{equation*}
commutes.
\end{enumerate}

\begin{lemma}\label{thisisadual}
The object $L(X^a)$ is a dual of $L(X)$ with coevalution $\eta:=\mathfrak{K}^{-1} \circ L_{\Delta_X}$ and evaluation $\varepsilon:=~\underset{X}{\cup}: L(X^a) \Lte L(X) \to k$.
\end{lemma}

\begin{proof}
Consider the diagram
\begin{equation*}
\xymatrix{
 L(X) \ar[r]^-{\eta \otimes \id}  \ar@{=}[d] &L(X) \Lte L(X^a) \Lte L(X) \ar[r]^-{\id \otimes \varepsilon} \ar[d]^-{\mathfrak{K} \otimes \id}& L(X) \ar@{=}[d]\\
L(X)\ar[r]_-{L_{\Delta_X} \otimes \id} & L(X \otimes X^a) \Lte L(X) \ar[r]_-{\underset{X^a}{\cup}}& L(X)
}
\end{equation*}

and the diagram

\begin{equation*}
\xymatrix{
 L(X^a) \ar[r]^-{\id \otimes \eta} \ar@{=}[d] &L(X^a) \Lte L(X) \Lte L(X^a) \ar[r]^-{\varepsilon \otimes \id} \ar[d]^-{\id \otimes \mathfrak{K}}& L(X^a) \ar@{=}[d]\\
L(X^a) \ar[r]_-{ \id \otimes L_{\Delta_{X^a}}}& L(X^a) \Lte L(X \otimes X^a) \ar[r]_-{\underset{X}{\cup}}& L(X^a).
}
\end{equation*}
These diagrams are made of two squares. The left squares commute by definition of $\eta$. The squares on the right commute because of the Property (P\ref{assol}). It follows that the two diagrams commute. Property (P\ref{idcomp}) implies that the bottom line of each diagram is equal to the identity. This proves the proposition.
\end{proof}

The preceding lemma shows that $L(X)$ is a dualizable object of $\Der^b(k)$. We set $L(X)^\star=\RHom_k(L(X),k)$. By Remark \ref{independance}, we have $L(X)^\star \simeq L(X^a)$.

Let $\lambda:k \to L(X \otimes X^a)$ be a morphism of $\Der^b(k)$. It defines a morphism
\begin{equation}\label{philambda}
\xymatrix{
\Phi_\lambda: L(X) \ar[r]^-{\lambda \Lte \id}& L(X \otimes X^a) \Lte L(X) \ar[r]^-{\underset{X}{\cup}}& L(X).\\
} 
\end{equation}

Consider the diagram

\begin{equation} \label{gros}
\xymatrix{%
 & L(X) \Lte L(X)^\star \ar[r]^-{\Phi_{\lambda} \otimes \id} &  L(X) \Lte L(X)^\star \ar[r]^-{\tau} & L(X)^\star \Lte L(X) \ar[rd]^-{\mathrm{ev}}&\\
k \ar[rd]_-{\eta}  \ar[ur]^-{\mathrm{coev}} &   &    &    & k\\
&L(X) \Lte L(X^a) \ar[r]^-{\Phi_{\lambda} \otimes \id}  & L(X) \Lte L(X^a) \ar[r]^-{\tau} & L(X^a) \Lte L(X) \ar[ru]_-{\varepsilon} &
}
\end{equation}

\begin{lemma} \label{maingros}
The diagram (\ref{gros}) commutes.
\end{lemma}

\begin{proof}
By Lemma \ref{thisisadual}, $L(X^a)$ is a dual of $L(X)$ with evaluation morphism $\varepsilon$ and coevaluation morphism $\eta$. It follows from Lemma \ref{invar} that the diagram (\ref{gros}) commutes.
\end{proof}

We identify $\lambda$ and the image of $1_k$ by $\lambda$ and similarly for $L_{\Delta_X}$. From now on, we write indifferently $\cup$ as a morphism or as an operation, as for example in Theorem \ref{formuleconclu}.

\begin{theorem}\label{formuleconclu}
Assuming properties (P\ref{compku}) to (P\ref{pairingdiag}), we have the formula

\begin{equation*}
\tr(\Phi_\lambda)= L_{\Delta_{X}} \underset{X^a \otimes X}{\cup} L(\sigma)\lambda.
\end{equation*}
If we further assume Property (P\ref{switch}) we have the formula
\begin{equation*}
\tr(\Phi_\lambda)= L_{\Delta_{X^a}} \underset{X \otimes X^a}{\cup} \lambda.
\end{equation*}
\end{theorem}

\begin{proof}
By definition of $\Phi_\lambda$, the diagram
\begin{small}
\begin{equation}\label{inter1}
\xymatrix @C=1pc{
& L(X)\Lte L(X^a) \ar[r]^-{\Phi_{\lambda} \otimes \id}  & L(X) \Lte L(X^a) \ar[r] & L(X^a) \Lte L(X) \ar[rd]^-{\varepsilon}&\\
k \ar[ru]^-{\eta} \ar[rd]_-{\lambda \otimes \eta} &&&&k\\
& L(X \otimes X^a) \Lte L(X) \Lte L(X^a) \ar[r]^-{\underset{X}{\cup} \otimes \id} & L(X) \Lte L(X^a) \ar[r] & L(X^a) \Lte L(X) \ar[ru]_-{\varepsilon}&.
}
\end{equation}
\end{small}
commutes.

Thus, computing the trace of $\Phi_\lambda$ is equivalent to compute the lower part of diagram (\ref{inter1}).

We denote by $\zeta$ the map
\begin{equation*}
\zeta: L(X^a \otimes X) \simeq k \Lte L(X^a \otimes X) \stackrel{L_{\Delta_X} \otimes \id}{\to} L(X \otimes X^a) \Lte L(X^a \otimes X) \stackrel{\underset{X^a \otimes X }{\cup}}{\to} k.
\end{equation*}
 
Consider the diagram
\begin{equation} \label{inter2}
\xymatrix{
&k \ar[ld]_-{\lambda \otimes \eta} \ar[rd]^-{\lambda \otimes L_{\Delta_X}} \ar @{}[d]|{1}&\\
L(X \otimes X^a) \Lte L(X) \Lte L(X^a) \ar[rr]^{ \id \otimes \mathfrak{K}} \ar[d]_-{\underset{X}{\cup} \Lte \id}&\ar @{}[d]|{2}& L(X \otimes X^a) \Lte L(X \otimes X^a) \ar[d]^{\underset{X}{\cup}} \\
L(X) \Lte  L(X^a) \ar[d] \ar[rr]^-{\mathfrak{K}}&\ar @{}[d]|{3}& L(X \otimes X^a) \ar[d]^-{L(\sigma)}\\
L(X^a) \Lte L(X) \ar[rr]^{\mathfrak{K}} \ar[rd]_-{\underset{X}{\cup}} &\ar @{}[d]|{4}& L(X^a \otimes X) \ar[ld]^{\zeta}\\
&k&. 
}
\end{equation}
This diagram is made of four sub-diagrams numbered from 1 to 4.
\begin{enumerate}
\item The sub-diagram 1 commutes by definition of $\eta$,
\item notice that $\mathfrak{K}= \underset{\ucat_\sC}{\cup}$ by the Property (P\ref{compku}). Then the sub-diagram 2 commutes by the Property (P\ref{assol}),
\item the sub-diagram 3 commutes because $L$ is a symmetric monoidal functor,
\item the sub-diagram 4 is the diagram of Property (P\ref{pairingdiag}).
\end{enumerate}

Applying Property (P\ref{idcomp}), we find that the right side of the diagram (\ref{inter2}) is equal to $L_{\Delta_X} \underset{X^a \otimes X}{\cup} L(\sigma)\lambda$.

By the Property (P\ref{switch}), $L_{\Delta_X} \underset{X^a \otimes X}{\cup} L(\sigma)\lambda= L(\sigma)L_{\Delta_X} \underset{X \otimes X^a}{\cup} \lambda$ and by  the Property (P\ref{viceversa}), $L(\sigma)L_{\Delta_X}=L_{\Delta_{X^a}}$, the result follows.
\end{proof}

\section{A short review on DQ-modules}

Deformation quantization modules have been introduced in \cite{Kos} and systematically studied in~\cite{KS3}.
We shall first recall here the main features of this theory, following the notation of loc.\ cit.

In all this paper, a manifold means a complex analytic manifold.
We denote by $\C^\hbar$ the ring $\C[[\hbar]]$. A Deformation Quantization algebroid stack (DQ-algebroid for short) on a complex manifold X with structure sheaf $\cO_X$, is a stack of $\C^\hbar$-algebras locally isomorphic to a star algebra $(\cO_X[[\hbar]],\star)$. If $\cA_X$ is a DQ-algebroid on a manifold $X$ then the opposite DQ-algebroid $\cA_X^{\textnormal{op}}$ is denoted by $\cA_{X^a}$. The diagonal embedding is denoted by $\delta_X: X \to X \times X$. 

If $X$ and $Y$ are two manifolds endowed with DQ-algebroids $\cA_X$ and $\cA_Y$, then $X \times Y$ is canonically endowed with the DQ-algebroid $\cA_{X \times Y}:=\cA_X \underline{\boxtimes} \cA_Y$ (see \cite[§2.3]{KS3}).
Following \cite[§2.3]{KS3}, we denote by $\cdot \boxtimes \cdot$ is the exterior product and by $\cdot \ubtimes \cdot$\index{1ub@$\cdot \ubtimes \cdot$} the bifunctor $\cA_{X \times Y} \underset{\cA_X \boxtimes \cA_Y}{\otimes} ( \cdot \boxtimes \cdot)$:
\begin{equation*}
\cdot \ubtimes \cdot : \Mod(\cA_X) \times \Mod(\cA_Y) \to \Mod(\cA_{X \times Y}).
\end{equation*}
We write $\cdot \ldetens \cdot$ for the corresponding derived bifunctor.

We write $\cC_X$\index{C@$\cC_X$} for the $\cA_{X \times X^a}$-module $\delta_{X \ast} \cA_X$ and $\omega_X \in \Mod_{\coh}(\cA_{X \times X^a})$\index{omegadualdq@$\omega_X$}  for the dualizing complex of DQ-modules. We denote by $\Du_{\cA_X}^\prime$ the duality functor of $\cA_X$-modules:
\begin{equation*}
\Du_{\cA_X}^\prime(\cdot):=\fRHom_{\cA_X}( \cdot,\cA_X).
\end{equation*}\index{dualb@$\mathbb{D}^\prime$}

Consider complex manifolds $X_i$ endowed with 
DQ-algebroids  $\A[X_i]$ ($i=1,2,\dots$).
\begin{Notation}\label{not:www1}
\begin{enumerate}[(i)]
\item
Consider a product of manifolds $X_1\times X_2 \times X_3$, we write it $X_{123}$. We denote by
$p_i$ the $i$-th projection and by $p_{ij}$ the $(i,j)$-th projection
({\em e.g.,} $p_{13}$ is the projection from 
$X_1\times X_1^a\times X_2$ to $X_1\times X_2$). 
We use similar notation for a product of four manifolds. 
\item
We write $\A[i]$ and $\A[ij^a]$
instead of $\A[X_i]$ and $\A[X_i\times X_j^a]$  and
similarly with other products. We use the same notations 
for $\dA[_i]$. 
\item
When there is no risk of confusion, we do note write the symbols 
$p_{i}^{-1}$ and similarly with $i$ replaced with $ij$, etc.\ 

\item If $\cK_1$ is an object of $\Der^b(\C^\hbar_{12})$ and $\cK_2$ is an object of $\Der^b(\C^\hbar_{23})$, we write $\cK_1 \underset{2}{\circ} \cK_2$ for $\dR p_{13!}(p_{12}^{-1} \cK_1 \Lte_{\C^\hbar_{123}} p_{23} ^{-1} \cK_2)$\index{1comp@$\circ$}.

\item We write $\Lte$ for the tensor product over $\C^\hbar$.
\end{enumerate}
\end{Notation}%

\subsection{Hochschild homology}

Let $X$ be a complex manifold endowed with a DQ-algebroid $\cA_X$. Recall that its 
 Hochschild homology is defined by

\begin{equation*}
\Hoc(\cA_X):= \delta_X^{-1}(\mathcal{C}_{X^a} \Lte_{\cA_{X \times X^a}} \mathcal{C}_X) \in \Der^b(\C_X^\hbar).
\end{equation*}\index{hochschildhomologie@$\Hoc$}
We denote by $\bHH(\cA_X)$ the object $\Rg(X,\Hoc(\cA_X))$ of the category $\Der^b(\C^\hbar)$ and by $\HH_0(\cA_X)$ the $\C^{\hbar}$-module   $\Hn^{0}(\bHH(\cA_X))$. We also set the notation, for a closed subset $\Lambda$ of $X$, $\Hoc_\Lambda(\cA_X):=\sGamma_\Lambda \Hoc(\cA_X)$ and $\HH_{0,\Lambda}(\cA_X)=\Hn^0(\Rg_\Lambda(X;\Hoc(\cA_X)))$.

\begin{proposition}
There is a natural isomorphism

\begin{equation}\label{isohocfais}
\Hoc(\cA_X)\simeq \fRHom_{\cA_{X \times X^a}} (\w_X^{-1}, \mathcal{C}_{X}).
\end{equation}
\end{proposition}

\begin{proof}
See \cite[§4.1, p.103]{KS3}.
\end{proof}

\begin{remark}
There is also a natural isomorphism
\begin{equation*}
\Hoc(\cA_X)\simeq \fRHom_{\cA_{X \times X^a}} (\cC_X, \w_X).
\end{equation*}
It can be obtain from the isomorphism (\ref{isohocfais}) by adjunction.
\end{remark}

\begin{proposition}[Künneth isomorphism] \label{kunisodq}
Let $X_i$ $(i=1, 2)$ be complexe manifolds endowed with DQ-algebroids $\cA_i$. 
\begin{enumerate}[(i)]
\item There is a natural morphism
\begin{equation} \label{kunnethorigin}
\fRHom_{\cA_{11^a}}(\w_1^{-1}, \cC_1) \overset{\mathrm{L}}{\boxtimes} \fRHom_{\cA_{22^a}}(\w_2^{-1}, \cC_2) \to \fRHom_{\cA_{121^a2^a}}(\w^{-1}_{12},\cC_{12}).
\end{equation}
\item If $X_1$ or $X_2$ is compact, this morphism induces a natural isomorphism
\begin{equation} \label{box morphism}
\mathfrak{K}: \bHH(\cA_1) \Lte  \bHH(\cA_2) \stackrel{\sim}{\to} \bHH(\cA_{12}).
\end{equation}\index{kunneth@$\mathfrak{K}$}
\end{enumerate}
\end{proposition}

\begin{proof}
(i) is clear. 

(ii) By \cite[Proposition 1.5.10]{KS3} and \cite[Proposition 1.5.12]{KS3}, the modules $\bHH(\cA_{i})$ for $(i=1, 2)$ and $\bHH(\cA_{12})$ are cohomologically complete.
If $X_1$ is compact, then the $\C^{\hbar}$-module $\bHH(\cA_{1})$ belongs to $\Der^b_{f}(\C^\hbar)$. Thus, the $\C^\hbar$-module $\bHH(\cA_{1}) \Lte_{\C^{\hbar}} \bHH(\cA_{2})$ is still a cohomologically complete module (see \cite[Proposition 1.6.5]{KS3}).

Applying the functor $\gr_{\hbar}$ to the morphism (\ref{box morphism}), we obtain the usual Künneth isomorphism for Hochschild homology of complex manifolds. Since $\gr_\hbar$ is a conservative functor on the category of cohomologically complete modules, the morphism (\ref{box morphism}) is an isomorphism.
\end{proof}

\subsection{Composition of Hochschild homology}\label{section_pairing}

Let $\Lambda_{ij} \;(i=1,2, j=i+1)$ be a closed subset of $X_{ij}$ and consider the hypothesis 

\begin{equation}\label{hypo}
p_{13} \mbox{ is proper on } \Lambda_{12} \times_{X_2} \Lambda_{23}.
\end{equation}

We also set $\Lambda_{12} \circ \Lambda_{23}= p_{13}(p_{12}^{-1} \Lambda_{12} \cap p_{23}^{-1} \Lambda_{23})$. 

\vspace{0.2cm}
\noindent Recall Proposition 4.2.1 of \cite{KS3}.

\begin{proposition}\label{pairing homology}
 Let $\Lambda_{ij}\; (i=1,2 \; j=i+1)$ satisfying (\ref{hypo}). There is a morphism
 \begin{equation} \label{pairingsheaves}
 \Hoc(\cA_{12^a}) \underset{2}{\circ}  \Hoc(\cA_{23^a}) \to \Hoc(\cA_{13^a}).
 \end{equation}
which induces a composition morphism for global sections
\begin{equation}\label{pairing}
\underset{2}{\cup}: \bHH_{\Lambda_{12}}(\cA_{{12^a}}) \Lte \bHH_{\Lambda_{23}}(\cA_{{23^a}}) \to \bHH_{\Lambda_{12} \circ \Lambda_{23}}(\cA_{{13^a}}).
\end{equation}\index{1comp@$\circ$}
\end{proposition}

\begin{corollary} \label{morph}
The morphism (\ref{pairingsheaves}) induces a morphism
\begin{equation}
\underset{\pt}{\cup}:\Hoc(\cA_1) \overset{\mathrm{L}}{\boxtimes} \Hoc(\cA_2) \to \Hoc(\cA_{12})
\end{equation}%
which coincides with the morphism (\ref{kunnethorigin}).
\end{corollary}

\begin{proof}
The result follows directly from the construction of morphism (\ref{pairingsheaves}). We refer the reader to \cite[§4.2]{KS3} for the construction.
\end{proof}%

We will state a result concerning the associativity of the composition of Hochschild homology. It is possible to compose kernels in the framework of DQ-modules. Here, we identify $X_1 \times X_2 \times X_{3^a}$ with the diagonal subset of $X_1 \times X_{2^a} \times X_2 \times X_{3^a}$.

The following definition is Defininition 3.1.2 and Definition 3.1.3 of \cite{KS3}.

\begin{definition}
Let $\cK_i\in\Der^b(\A[ij^a])$ ($i=1,2$, $j=i+1$). 
One sets
\begin{align*}
\cK_1 \uLte_{\cA_2} \cK_2 &= 
(\cK_1\ldetens\cK_2)\Lte_{\A[22^a]} \dA[2]\\
&=p_{12}^{-1} \cK_1 \Lte_{p_{12}^{-1}\cA_{1^a2}} \cA_{123} \Lte_{p_{23^a}^{-1} \cA_{23^a}} p_{23}^{-1} \cK_2,\\
\cK_1\conv[X_2]\cK_2 &= 
\reim{p_{14}}\bl(\cK_1\ldetens\cK_2)\Lte_{\A[22^a]} \dA[2]\br,\\
\cK_1\sconv[X_2]\cK_2&= 
\roim{p_{14}}\bl(\cK_1\ldetens\cK_2)\Lte_{\A[22^a]} \dA[2]\br.
\end{align*}\index{1utens@$\uLte$}\index{1comp@$\circ$}
\end{definition}

It should be noticed that $\uLte$, $\circ$ and $\ast$ are not associative in general. 
\begin{remark}\label{remasso}
There is a morphism $\cK_1 \Lte_{\cA_2} \cK_2 \to \cK_1 \uLte_{\cA_2} \cK_2$ which is an isomorphism if $X_1=\pt$ or $X_3=\pt$. 
\end{remark}

The following proposition, which corresponds to {\cite[Proposition 3.2.4]{KS3}}, states a result concerning the associativity of the composition of kernels in the category of DQ-modules and will be useful for the sketch of proof of Proposition \ref{associatif}.
\begin{proposition}\label{assodqmod}
Let $\mathcal{K}_i \in \Der_{coh}^b(\cA_{i (i+1)^a})$ $(i=1, \, 2, \, 3)$ and let $\mathcal{L} \in \Der_{\coh}^b(\cA_{4})$. Set $\Lambda_i=\Supp(\cK_i)$ and assume that $\Lambda_i \times_{X_{i+1}} \Lambda_{i+1}$ is proper over $X_i \times X_{i+2}$ (i=1, 2).
\begin{enumerate}[(i)]
\item There is a canonical isomorphism $(\cK_1 \underset{2}{\circ} \cK_2) \ldetens \mathcal{L} \stackrel{\sim}{\to} \cK_1 \underset{2}{\circ}( \cK_2 \ldetens \mathcal{L})$.

\item There are canonical isomorphisms
\begin{equation*}
(\cK_1 \underset{2}{\circ} \cK_2) \underset{3}{\circ} \cK_3 \stackrel{\sim}{\leftarrow} (\cK_1 \ldetens \cK_2 \ldetens \cK_3)\underset{22^a33^a}{\circ}(\cC_2 \ldetens \cC_3) \stackrel{\sim}{\to}\cK_1 \underset{2}{\circ} (\cK_2 \underset{3}{\circ} \cK_3).
\end{equation*}
\end{enumerate}
\end{proposition}
The next proposition is the translation of Property (P\ref{assol}) in the framework of DQ-modules.
\begin{proposition}\label{associatif}
\begin{enumerate}[(i)]
\item \label{claimasso} Assume that $X_i$ is compact for $i= 2, \; 3$. The following diagram is commutative
\begin{small}
\begin{equation*}
\xymatrix @C=1.2pc{\Hoc(\cA_{12^a}) \underset{2}{\circ}  \Hoc(\cA_{23^a}) \underset{3}{\circ}  \Hoc(\cA_{34^a}) \ar[r] \ar[d] & \Hoc(\cA_{12^a}) \underset{2}{\circ}  \Hoc(\cA_{24^a}) \ar[d] \\
\Hoc(\cA_{13^a}) \underset{3}{\circ}  \Hoc(\cA_{34^a}) \ar[r] & \Hoc(\cA_{14^a}).}
\end{equation*}
\end{small}
\item Assume that $X_i$ is compact for $i=1, \; 2, \; 3, \; 4$. The preceding diagram induces a commutative diagram
\begin{small}
\begin{equation*}
\xymatrix{\bHH(\cA_{12^a}) \Lte  \bHH(\cA_{23^a}) \Lte \bHH(\cA_{34^a}) \ar[r] \ar[d] & \bHH(\cA_{12^a}) \Lte\ \bHH(\cA_{24^a}) \ar[d] \\
\bHH(\cA_{13^a}) \Lte  \bHH(\cA_{34^a}) \ar[r] & \bHH(\cA_{14^a}).}
\end{equation*}
\end{small}
\end{enumerate}
\end{proposition}

\begin{proof}[Sketch of Proof]
\begin{enumerate}[(i)]
\item If $\cM \in \Der(\cA_X)$ and $\cN \in \Der(\cA_Y)$, we write $\cM \cN$ for  $\cM \underline{\boxtimes} \cN$ and $i^k$ for $\underbrace{X_i \times \ldots \times X_i}_{k \; times}$. For the legibility, we omit the upper script $(\cdot)^a$ when indicating the base of a composition.

Following the notation of \cite[§4.2]{KS3}, we set $S_{ij}:= \w_i^{-1} \underline{\boxtimes} \cC_{j^a} \in \Der_{\coh}^b(\cA_{ii^aj^aj})$ and $K_{ij}=\cC_i \underline{\boxtimes} \w_{j^a} \in \Der_{\coh}^b(\cA_{ii^aj^aj})$. It follows that%
\begin{equation*}
\Hoc(\cA_{ij^a})\simeq \fRHom_{\cA_{ii^aj^aj}}(S_{ij},K_{ij}).
\end{equation*}%

We deduce from Proposition \ref{assodqmod} (ii), the following diagram which commutes.

\begin{equation}\label{assoundemi}
\scalebox{0.61}{
\xymatrix @C=1.2pc{\Hoc(\cA_{12^a}) \underset{2}{\circ}  \Hoc(\cA_{23^a}) \underset{3}{\circ}  \Hoc(\cA_{34^a}) \ar[r] \ar[d] & \fRHom_{}(S_{12} \underset{2^2}{\circ} S_{23}, K_{12} \underset{2^2}{\circ} K_{23})  \underset{3}{\circ} \Hoc(\cA_{24^a}) \ar[d] \\
\Hoc(\cA_{12^a}) \underset{2}{\circ}  \fRHom_{}(S_{23} \underset{3^2}{\circ} S_{34}, K_{23} \underset{3^2}{\circ} K_{34}) \ar[d] & \fRHom_{}((S_{12} \underset{2^2}{\circ} S_{23}) \underset{3^2}{\circ} S_{34}, (K_{12} \underset{2^2}{\circ}K_{23}) \underset{3^2}{\circ} K_{34}) \ar[d]\\
\fRHom(S_{12} \underset{2^2}{\circ} (S_{23} \underset{3^2}{\circ} S_{34}), K_{12} \underset{2^2}{\circ} (K_{23} \underset{3^2}{\circ} K_{34})) \ar[r] & \fRHom((S_{12} S_{23}  S_{34})\underset{2^43^4}{\circ}(\cC_{22^a}   \cC_{33^a}),(K_{12}  K_{23}  K_{34})\underset{2^43^4}{\circ}(\cC_{22^a} \cC_{33^a})).}}
\end{equation}
Following the proof of \cite[Proposition 4.2.1]{KS3}, we have a morphism
\begin{equation} \label{integrationpairing}
K_{ij} \underset{j^2}{\circ} K_{jk} \to K_{ik}
\end{equation}
constructed as follows
\begin{align*}
(\cC_i \w_{j^a}) \uLte_{\cA_{jj^a}} (\cC_j \w_{k^a}) \simeq& ((\cC_i \w_{j^a}) (\cC_j \w_{k^a})) \Lte_{\cA_{jj^a(jj^a)^a}} \cC_{jj^a}\\
\simeq & ((\cC_i \w_{k^a}) (\w_{j^a} \cC_j)) \Lte_{\cA_{jj^a(jj^a)^a}} \cC_{jj^a}\\
\simeq & (\cC_i \w_{k^a} \w_{j^a}) \Lte_{\cA_{jj^a}} \cC_j\\
\to & \lbrack (\cC_i \w_k) p_j^{-1} \delta_{j \ast} \Omega_j^{\cA} \rbrack \Lte_{\mathcal{D}_j^{\cA}} p_j^{-1} \delta_{\ast j} \cA_j \stackrel{\sim}{\leftarrow} p_{ik}^{-1}(\cC_i \w_k) [2d_j].
\end{align*}

where $\mathcal{D}_j^{\cA}$ is the quantized ring of differential operator with respect to $\cA_j$ (see Definition 2.5.1 of \cite{KS3}) and $\Omega_j^{\cA}$ is the quantized module of differential form with respects to $\cA_j$ (see Definition 2.5.5 of \cite{KS3}). By \cite[Lemma 2.5.5]{KS3} there is an isomorphism $\Omega_j^{\cA} \Lte_{\mathcal{D}_j^{\cA}} \cA_j[-d_j] \simeq \C^\hbar_j$ where $d_j$ denotes the complex dimension of $X_j$. This isomorphism gives the last arrow in the construction  of morphism (\ref{integrationpairing}).

By adjunction between $Rp_{ik!}$ and $p_{ik}^!\simeq p_{ik}^{-1}[2d_j]$ , we get the morphism (\ref{integrationpairing}).
Choosing $i=1$, $j=23$ and $k=4$, we get the morphism
\begin{equation*}
(\cC_1 \w_{4^a} \w_{2^a3^a})\underset{2^23^2}{\circ} \cC_{23} \to \cC_1 \w_{4^a}.
\end{equation*}
There are the isomorphisms
\begin{align*}
(K_{12}  K_{23}  K_{34})\underset{2^43^4}{\circ}(\cC_{22^a} \cC_{33^a}) & \simeq ((\cC_1 \w_{4^a}\w_{2^a3^a}) \cC_{23}) \underset{2^43^4}{\circ}(\cC_{232^a3^a})\\
& \simeq (\cC_1 \w_{4^a} \w_{2^a3^a})\underset{2^23^2}{\circ} \cC_{23}.
\end{align*}
Thus, we get a map
\begin{equation*}
(K_{12}  K_{23}  K_{34})\underset{2^43^4}{\circ}(\cC_{22^a} \cC_{33^a}) \to K_{14}.
\end{equation*}
By construction of the morphism (\ref{integrationpairing}) and of the isomorphism of Proposition \ref{assodqmod} (ii), the below diagram commutes
\begin{equation} \label{assoK}
\xymatrix{ (K_{12} \underset{2^2}{\circ} K_{23})\underset{3^2}{\circ} K_{34} \ar[r]  & K_{13} \underset{3^2}{\circ} K_{34} \ar[d] \\
(K_{12}  K_{23}  K_{34})\underset{2^43^4}{\circ}(\cC_{22^a}\cC_{33^a}) \ar[d]^-{\wr}  \ar[u]_-{\wr} \ar[r]& K_{14}\\
K_{12} \underset{2^2}{\circ} (K_{23}\underset{3^2}{\circ} K_{34}) \ar[r]  & K_{12} \underset{2^2}{\circ} K_{24}. \ar[u]
} 
\end{equation}

Similarly, we get the following commutative diagram

\begin{equation}\label{assoS}
\xymatrix{ S_{13} \underset{3^2}{\circ} S_{34} \ar[r]& S_{12} \underset{2^2}{\circ}(S_{23} \underset{3^2}{\circ} S_{34})  \\
S_{14} \ar[u] \ar[d] \ar[r] & (S_{12} S_{23} S_{34}) \underset{2^4 3^4}{\circ} ( \cC_{22^a} \cC_{33^a}) \ar[u]_-{\wr} \ar[d]^-{\wr} \\
S_{12} \underset{2^2}{\circ} S_{24} \ar[r]& (S_{12} \underset{2^2}{\circ}S_{23}) \underset{3^2}{\circ} S_{34}. 
}
\end{equation}

It follows from the commutation of the diagrams (\ref{assoK}) and (\ref{assoS}) that the diagram below commutes.
\begin{equation}\label{assodeuxdemi}
\scalebox{0.7}{
\xymatrix{\fRHom(S_{12} \underset{2^2}{\circ} S_{23} , K_{12} \underset{2^2}{\circ} K_{23})\underset{3}{\circ} \Hoc(\cA_{34^a}) \ar[r] \ar[d] & \Hoc(\cA_{13^a}) \underset{3}{\circ} \Hoc(\cA_{34^a}) \ar[d]\\
\fRHom((S_{12} \underset{2^2}{\circ} S_{23}) \underset{3^2}{\circ} S_{34}, (K_{12} \underset{2^2}{\circ} K_{23}) \underset{3^2}{\circ} K_{34}) \ar[r]  & \fRHom(S_{13} \underset{3^2}{\circ} S_{34} , K_{13} \underset{3^2}{\circ} K_{34}) \ar[d]\\
\fRHom((S_{12} S_{23}  S_{34})\underset{2^43^4}{\circ}(\cC_{22^a}   \cC_{33^a}),(K_{12}  K_{23}  K_{34})\underset{2^43^4}{\circ}(\cC_{22^a} \cC_{33^a})) \ar[u]^-{\wr} \ar[d]_-{\wr} \ar[r] & \Hoc(\cA_{14})\\
\fRHom(S_{12} \underset{2^2}{\circ} (S_{23} \underset{3^2}{\circ} S_{34}), K_{12} \underset{2^2}{\circ} (K_{23} \underset{3^2}{\circ} K_{34})) \ar[r]  & \fRHom(S_{12} \underset{2^2}{\circ} S_{24} , K_{12} \underset{2^2}{\circ} K_{24}) \ar[u]\\
\Hoc(\cA_{12}) \underset{2}{\circ} \fRHom(S_{23} \underset{3^2}{\circ} S_{34}, K_{23} \underset{3^2}{\circ} K_{34}) \ar[r] \ar[u] & \Hoc(\cA_{12}) \underset{2}{\circ} \Hoc(\cA_{24}). \ar[u]}
}
\end{equation}

The commutativity of the diagram (\ref{assoundemi}) and (\ref{assodeuxdemi}) prove (\ref{claimasso}).

\item is a consequence of (i) and of Proposition \ref{kunisodq} (ii).
\end{enumerate}
\end{proof}
\subsection{Hochschild class}

Let $\mathcal{M} \in \Der^b_{\coh}(\cA_X)$. We have the chain of morphisms

\begin{align*}
\hh_{\mathcal{M}}:\fRHom_{\cA_X}(\cM,\cM) &\stackrel{\sim}{\leftarrow} \Du_{\cA_X}^{\prime}(\cM) \Lte_{\cA_X} \cM\\
               & \simeq \delta^{-1}(\mathcal{C}_{X^a} \Lte_{\cA_{X \times X^a}} ( \cM \underline{\boxtimes} \Du^{\prime}_{\cA_X}(\cM)))\\
               & \rightarrow \delta^{-1}(\mathcal{C}_{X^a} \Lte_{\cA_{X \times X^a}} \mathcal{C}_X).
\end{align*}

We get a map 

\begin{equation}\label{map hoch}
\hh^0_\cM:\Hom_{\cA_X}(\cM,\cM) \to \Hn^0_{\Supp(\cM)}(X,\Hoc(\cA_X)).
\end{equation}

\begin{definition}
The image of an endomorphism $f$ of $\mathcal{M}$ by the map (\ref{map hoch}) gives an element $\hh_{X}(\cM,f) \in \Hn^0_{\Supp(\cM)}(X,\Hoc(\cA_X))$\index{hochschildclass@$\hh$} called the Hochschild class of the pair $(\cM,f)$. If $f=\id_{\cM}$, we simply write $\hh_X(\cM)$\index{hochschildclass@$\hh$} and call it the Hochschild class of $\cM$.
\end{definition}

\begin{remark}\label{representation_Ch}
Let $M \in \Der^{b}_{f}(\C^\hbar)$ and let $f \in \Hom_{\C^\hbar}(M,M)$. Then the Hochschild class $\hh_{\C^{\hbar}}(M,f)$ of $f$ is obtained by the composition
\begin{align*}
  \C^\hbar \to \RHom_{\C^\hbar}(M,M) 
\to & M \Lte_{\C^\hbar} \RHom_{\C^\hbar}(M,\C^\hbar)  \stackrel{f \otimes \id}{\to} 
M \Lte_{\C^\hbar} \RHom_{\C^\hbar}(M,\C^\hbar)& \\
&\to \RHom_{\C^\hbar}(M,\C^\hbar) \Lte_{\C^\hbar} M \to \C^\hbar . 
\end{align*}
Thus, it is the trace of $f$ in $\Der^b(\C^\hbar)$.
\end{remark}

\subsection{Actions of Kernels}

We explain how kernels act on Hochschild homology. 
Let $X_1$ and $X_2$ be compact complex manifolds endowed with DQ-algebroids $\cA_1$ and $\cA_2$. Let $\lambda \in \HH_0(\cA_{12^a})$. There is a morphism 
\begin{equation} \label{moratr}
\Phi_\lambda: \bHH(\cA_2) \to \bHH(\cA_1)
\end{equation}%
given by%
\begin{equation*}
\bHH(\cA_2) \simeq \C^\hbar \Lte \bHH(\cA_2) \stackrel{\lambda \otimes \id}\to \bHH(\cA_{12^a}) \Lte \bHH(\cA_2) \stackrel{\underset{2}{\cup}}{\to} \bHH(\cA_1).
\end{equation*}%
If $\cK$ is an object of $\Der^b_{\coh}(\cA_{12^a})$ then there is a morphism
\begin{equation} \label{Phik}
\Phi_{\cK}: \bHH(\cA_{2}) \to \bHH_(\cA_{1})
\end{equation}%
obtained from morphism (\ref{moratr}) by choosing $\lambda=\hh_{X_{12^a}}(\cK)$. In \cite{KS3}, the authors give initially a different definition and show in \cite[Lemma 4.3.4]{KS3} that it is equivalent to the present definition. 

We denote by $\w^\mathrm{top}_X$\index{omegadualtop@$\w^\mathrm{top}_X$} the dualizing complex of the category $\Der^{+}(\C_X^{\hbar})$.

\begin{proposition} \label{prep}
Let $X_i$, $(i=1, \,2)$ be a compact complex manifold endowed with a DQ-algebroid $\cA_i$. 

\begin{enumerate}[(i)]

\item The following diagram commutes.

\begin{equation}
\xymatrix{
p_1^{-1} \Hoc(\cA_{1^a}) \Lte \Hoc(\cA_{12^a}) \Lte p_2^{-1} \Hoc(\cA_2) \ar[d] \ar[r]^-{\cdot \underset{1}{\cup} \cdot \underset{2}{\cup} \cdot}& \w^{\mathrm{top}}_{12} \\
\Hoc(\cA_{12^a}) \Lte \Hoc(\cA_{1^a2}) \ar[ru]_-{\underset{1^a2}{\cup}} &.
}
\end{equation}

\item The diagram
\begin{equation}\label{diagprep}
\xymatrix{
 \bHH(\cA_{1^a}) \Lte \bHH(\cA_{12^a}) \Lte \bHH(\cA_2) \ar[d] \ar[r]^-{\cdot \underset{1}{\cup} \cdot \underset{2}{\cup} \cdot}& \C^\hbar \\
\bHH(\cA_{12^a}) \Lte \bHH(\cA_{1^a2}) \ar[ru]_-{\underset{1^a2}{\cup}} & 
}
\end{equation}
commutes.
\end{enumerate}
\end{proposition}

\begin{proof}
\begin{enumerate}[(i)]

\item \label{premierpoint} In view of Remark \ref{remasso}, only usual tensor products are involved. Thus, it is a consequence of the projection formula and of the associativity of the tensor product.

\item follows from (\ref{premierpoint}).
\end{enumerate}
\end{proof}%

The composition
\begin{equation*}
\C^\hbar_{X \times X^a} \to \fRHom_{\cC_{X \times X^a}}(\cC_X,\cC_X)\stackrel{\hh_{\cC_X}}{\to} \Hoc(X \times X^a)
\end{equation*}
induces a map
\begin{equation}\label{hhdelta}
\hh(\Delta_X):\C^\hbar \to \bHH(\cA_{X \times X^a}).
\end{equation}%
The image of $1_{\C^\hbar}$ by $\hh(\Delta_X)$ is $\hh_{X \times X^a}(\cC_X)$.

\begin{proposition}\label{preuveidcomp}
The left (resp. right) actions of $\hh_{X \times X^a}(\cC_X)$ on $\bHH(\cA_X)$ (resp. $\bHH(\cA_{X^a})$) via the morphism (\ref{pairing}) are the trivial action.
\end{proposition}

\begin{proof}
See \cite[Lemma 4.3.2]{KS3}.
\end{proof}

We define the morphism $\zeta:\bHH(\cA_{X \times X^a}) \to \C^\hbar$ as the composition
\begin{equation*}
\scalebox{0.91}{$
\bHH(\cA_{X^a \times X}) \simeq \C^\hbar \Lte \bHH(\cA_{X^a \times X}) \stackrel{\hh(\Delta_X) \otimes \id}{\to} \bHH(\cA_{X \times X^a}) \Lte \bHH(\cA_{X^a \times X}) \stackrel{\underset{X^a \times X}{\cup}}{\longrightarrow} \C^\hbar.
$}
\end{equation*}

\begin{corollary} \label{ld}
Let $X$ be a compact complex manifold endowed with a DQ-algebroid $\cA_X$. The  diagram below commutes.
\begin{equation*}
\xymatrix{
\bHH(\cA_{X^a}) \Lte \bHH(\cA_X) \ar[r]^-{\underset{X}{\cup}} \ar[d]^-{\mathfrak{K}}& \C^\hbar\\
\bHH(\cA_{X^a \times X}) \ar[ru]_-{\zeta}.
}
\end{equation*}
\end{corollary}

\begin{proof}
It follows from Proposition \ref{prep} with $X_1=X_2=X$, that the triangle on the right of the below diagram commutes. The commutativity of the square on the left is tautological.
\begin{equation*}
\scalebox{0.9}{\xymatrix{
\bHH(\cA_{X^a}) \Lte \C^\hbar \Lte \bHH(\cA_X) \ar[r]^-{\id \otimes \hh(\Delta_X) \otimes \id} \ar[d]^{\wr} &\bHH(\cA_{X^a}) \Lte \bHH(\cA_{X \times X^a}) \Lte \bHH(\cA_X) \ar[d] \ar[r]^-{\cdot \underset{X}{\cup} \cdot \underset{X}{\cup} \cdot}& \C^\hbar \\
\C^\hbar \Lte \bHH(\cA_{X^a}) \Lte \bHH(\cA_X) \ar[r]^-{\hh(\Delta_X) \otimes \mathfrak{K}}&\bHH(\cA_{X \times X^a}) \Lte \bHH(\cA_{X^a \times X}) \ar[ru]_-{\underset{X^a \times X}{\cup}} & 
}}
\end{equation*}%
\end{proof}

Finally, an important result is the Theorem 4.3.5 of \cite{KS3}: 

\begin{theorem}\label{Riemann-Roch}
Let $\Lambda_i$ be a closed subset of $X_i \times X_{i+1} \; (i=1 \; 2)$ and assume that $\Lambda_1 \times_{X_2} \Lambda_2$ is proper over $X_1 \times X_3$. Set $\Lambda= \Lambda_1 \circ \Lambda_2$. Let $\cK_i \in \Der^b_{\coh, \Lambda_i}(\cA_{X_i \times X^a_{i+1}}) (i=1,\; 2)$. Then

\begin{equation}
\hh_{X_{13^a}}(\cK_1 \circ \cK_2)=\hh_{X_{12^a}}(\cK_1) \underset{2}{\cup} \hh_{X_{23^a}}(\cK_2)
\end{equation}
as elements of $\HH_{\Lambda}^0(\cA_{X_1 \times X_{3^a}})$.
\end{theorem}

\begin{proof}
See \cite[p. 111]{KS3}.
\end{proof}

\section{A Lefschetz formula for DQ-modules}

\subsection{The monoidal category of DQ-algebroid stacks}\label{mono}

In this subsection we collect a few facts concerning the product $\cdot \underline{\boxtimes} \cdot$ of DQ-algebroids. Recall that if $X$ and $Y$ are two complex manifolds endowed with DQ-algebroids $\cA_X$ and $\cA_Y$, $X \times Y$ is canonically endowed with the DQ-algebroid $\cA_{X \times Y}:= \cA_X \underline{\boxtimes} \cA_Y$. There is a functorial symmetry isomorphism
\begin{align*}
\sigma_{X,Y}& : (X \times Y, \cA_{X \times Y}) \stackrel{\sim}{\to}  (Y \times X, \cA_{Y \times X})\\
\end{align*}%
and for any triple $(X, \cA_X)$, $(Y,\cA_Y)$ and $(Z, \cA_Z)$ there is a natural associativity isomorphism 
\begin{equation*}
\rho_{X,Y,Z}: (\cA_X \ubtimes \cA_Y) \ubtimes \cA_Z \stackrel{\sim}{\to} \cA_X \ubtimes (\cA_Y \ubtimes \cA_Z).
\end{equation*}
 
We consider the category $\mathscr{DQ}$ whose objects are the pairs $(X,\cA_X)$ where X is a complex manifold and $\cA_X$ a DQ-algebroid stack on $X$ and where the morphisms are obtained by composing and tensoring the identity morphisms, the symmetry morphisms and the associativity morphisms. The category $\mathscr{DQ}$ endowed with $\ubtimes$ is a symmetric monoidal category.

We denote by
\begin{small}
\begin{equation*}
v: ((X \times Y) \times (X \times Y)^a, \cA_{(X \times Y) \times (X \times Y)^a}) \to ((Y \times X) \times (Y \times X)^a,\cA_{(Y \times X) \times (Y \times X)^a}) 
\end{equation*}
\end{small}
the map defined by $v:=\sigma \times \sigma$. 

In this situation, after identifying, $(X \times X^a) \times (Y \times Y^a)$ with $(X \times Y) \times (X \times Y)^a$, there is a natural isomorphism $\cC_X \underline{\boxtimes} \cC_Y \simeq \cC_{X \times Y}$ and the morphism $v$ induces an isomorphism
\begin{equation*}
v_\ast(\cC_{X \times Y}) \simeq \cC_{Y \times X}.
\end{equation*}
\begin{proposition}
The map $\sigma_{X,Y}$ induce an isomorphism
\begin{equation} \label{morphismedecoherenceinv}
\sigma_{\ast}: \sigma_{X,Y \ast}(\Hoc(\cA_{X \times Y})) \to \Hoc(\cA_{Y \times X})
\end{equation}
\end{proposition}

\begin{proof}
There is the following Cartesian square of topological space.
\begin{equation*}
\xymatrix{X \times Y \ar[d]^-{\delta_{X \times Y}} \ar[r]^-{\sigma}& Y \times X \ar[d]^-{\delta_{Y \times X}}\\
(X \times Y) \times (X \times Y) \ar[r]^-{v}& (Y \times X) \times (Y \times X).
}
\end{equation*}
Then 
\begin{align*}
\sigma_\ast \Hoc(\cA_{X \times Y})&\simeq \sigma_! \delta_{X \times Y}^{-1} (\cC_{(X \times Y)^a} \Lte_{\cA_{X \times Y}} \cC_{X \times Y})\\
                             &\simeq \delta_{Y \times X}^{-1} v_! (\cC_{(X \times Y)^a} \Lte_{\cA_{X \times Y}} \cC_{X \times Y})\\
                             &\simeq \delta^{-1}_{Y \times X}(\cC_{(Y \times X)^a} \Lte_{\cA_{Y \times X}} \cC_{Y \times X}).
\end{align*}
\end{proof}
The morphsim (\ref{morphismedecoherenceinv}) induces an isomorphism that we still denote $\sigma_\ast$
\begin{equation*}
\sigma_\ast :\bHH(\cA_{X \times Y}) \to \bHH(\cA_{Y \times X}).
\end{equation*}
The following diagram commutes
\begin{equation}
\xymatrix{ 
\bHH(\cA_{X \times Y}) \ar[r]^-{\sigma_\ast}& \bHH(\cA_{Y \times X})\\
\bHH(\cA_X) \Lte \bHH(\cA_Y) \ar[u]^-{\mathfrak{K}} \ar[r]& \bHH(\cA_Y) \Lte \bHH(\cA_X) \ar[u]^-{\mathfrak{K}}.
}
\end{equation}

\begin{proposition}\label{propiii} 
There is the equality
\begin{equation*}
\sigma_\ast\hh_{X \times X^a}(\cC_X)=\hh_{X^a \times X}(\cC_{X^a}).
\end{equation*}
\end{proposition}

\begin{proof}
Immediate by using Lemma 4.1.4 of \cite{KS3}. 
\end{proof}
\subsection{The Lefschetz-Lunts formula for DQ-modules}

Inspired by the Lefschetz formula for Fourier-Mukai functor of V. Lunts (see \cite{Lunts}), we give a similar formula in the framework of DQ-modules. 

\begin{theorem} \label{preLunts}
Let $X$ be a compact complex manifold equiped with a DQ-algebroid $\cA_X$. Let $\lambda \in \HH_0(\cA_{X \times X^a})$. Consider the map (\ref{moratr}) 
\begin{equation*}
\Phi_\lambda: \bHH(\cA_X) \to \bHH(\cA_X).
\end{equation*}
Then
\begin{equation*}
\tr_{\C^\hbar}(\Phi_\lambda)= \hh_{X^a \times X}(\cC_{X^a}) \underset{X \times X^a}{\cup} \lambda.
\end{equation*}
\end{theorem}

\begin{proof}
Consider the full subcategory $\sC$ of $\mathscr{DQ}$ whose objects are the pair $(X,\cA_X)$ where $X$ is a compact manifold. By the results of Subsection \ref{mono}, the pair $(\bHH,\mathfrak{K})$ is a symmetric monoidal functor.

The data are given by

\begin{enumerate}[(a)]
\item the functor $(\cdot)^a$ which associate to a DQ-algebroid $(X,\cA_X)$ the opposite DQ-algebroid $(X,\cA_{X^a})$,
\item the monoidal functor on $\sC$ given by the pair $(\bHH, \mathfrak{K})$,
\item the morphism (\ref{pairing}),
\item for each pair $(X, \cA_X)$ the morphism $\hh(\Delta_X)$. 
\end{enumerate}
  
We check the properties requested by of our formalism:
\begin{enumerate}[(i)]
\item the Property (P\ref{compku}) is granted by Corollary \ref{morph},
\item the Property (P\ref{assol}) follows from Proposition \ref{associatif},
\item the Property (P\ref{viceversa}) follows from Proposition \ref{propiii},
\item the Property (P\ref{idcomp}) follows from Proposition \ref{preuveidcomp},
\item the Property (P\ref{pairingdiag}) follows from Proposition \ref{ld},
\item the Property (P\ref{switch}) follows from the construction of the pairing.
\end{enumerate}
Then the formula follows from Theorem \ref{formuleconclu}.
\end{proof}

\begin{corollary}\label{almostfinal}
Let $X$ be a compact complex manifold endowed with a DQ-algebroid $\cA_X$ and let $\cK \in \Der^b_{\coh}(\cA_{X \times X^a})$. Then
\begin{equation*}
\tr_{\C^\hbar}(\Phi_\cK)= \hh_{X^a \times X}(\cC_{X^a}) \underset{X \times X^a}{\cup} \hh_{X \times X^a}(\cK).
\end{equation*}
\end{corollary}

\begin{proof}
Apply Theorem \ref{preLunts} to $\Phi_\cK$.
\end{proof}

\begin{corollary}\label{Lunts}
Let $X$ be a compact complex manifold endowed with a DQ-algebroid $\cA_X$ and let $\cK \in \Der^b_{\coh}(\cA_{X \times X^a})$. Then
\begin{equation*}
\tr_{\C^\hbar}(\Phi_\cK)=\chi(\Rg(X \times X^a; \mathcal{C}_{X^a} \Lte_{\cA_{X \times X^a}} \cK)).
\end{equation*}
\end{corollary}

\begin{proof}
By Corollary \ref{almostfinal}, we get that
\begin{equation*}
\tr_{\C^\hbar}(\Phi_\cK)= \hh_{X^a \times X}(\cC_{X^a}) \underset{X \times X^a}{\cup} \hh_{X \times X^a}(\cK).
\end{equation*}

Applying Theorem \ref{Riemann-Roch} with $X_1=X_3=\pt$ and $X_2=X \times X^a$ we find that
\begin{equation*}
\hh_{X^a \times X}(\cC_{X^a}) \underset{X \times X^a}{\cup} \hh_{X \times X^a}(\cK)=\hh_{\pt}(\Rg(X\times X^a; \mathcal{C}_{X^a} \Lte_{\cA_{X \times X^a}} \cK).
\end{equation*} 
By Remark \ref{representation_Ch}, it follows that 
\begin{equation*}
\hh_{\pt}(\Rg(X\times X^a; \mathcal{C}_{X^a} \Lte_{\cA_{X \times X^a}} \cK)= \chi(\Rg(X \times X^a; \mathcal{C}_{X^a} \Lte_{\cA_{X \times X^a}} \cK)). 
\end{equation*}
Finally, we get that $\tr_{\C^\hbar}(\Phi_\cK)=\chi(\Rg(X\times X^a; \mathcal{C}_{X^a} \Lte_{\cA_{X \times X^a}} \cK))$.
\end{proof}

\subsection{Applications}
We give some consequences of Theorem \ref{preLunts} and explain how to recover some of the results of the paper \cite{Lunts} of V. Lunts and give a special form of the formula when $X$ is also symplectic.

\begin{theorem}[\cite{Lunts}] \label{lunts1}
Let $X$ be a compact complex manifold and $\cK$ an object of $\Der^b_{\coh}(\cO_{X \times X})$. Then,
\begin{equation*}
\sum_i (-1)^i \tr(\Hn^i(\Phi_\cK))=\chi(\Rg(X \times X ; \mathcal{O}_{X} \Lte_{\cO_{X \times X}} \cK)).
\end{equation*}
\end{theorem}

\begin{proof}
We endow $X$ with the trivial deformation. Then, we can apply Corollary \ref{Lunts} and forget $\hbar$ by applying $\gr_\hbar$. We recover Theorem 3.9 of \cite{Lunts}.
\end{proof}

\begin{proposition}
Let $X$ be a compact complex manifold endowed with a DQ-algebroid $\cA_X$ and let $\cK \in \Der^b_{\coh}(\cA_{X \times X^a})$. Then
\begin{equation*}
\tr(\Phi_\cK)=\tr(\Phi_{\gr_\hbar \cK}).
\end{equation*}
\end{proposition}

\begin{proof}
Remark that 
\begin{equation*}
\chi(\RHom_{\cA_X}(\w_X^{-1}, \cK))=\chi(\RHom_{\gr_\hbar \cA_X}((\gr_\hbar \w_X^{-1}), \gr_\hbar \cK)). 
\end{equation*}
Then, the result follows by Corollary \ref{Lunts} and Theorem \ref{lunts1}.
\end{proof}

It is possible to localize $\cA_X$ with respect to $\hbar$. We denote by $\C((\hbar))$ the field of formal Laurent series. We set $\cA_X^{loc}=\C((\hbar)) \otimes \cA_X$. If $\cM$ is a $\cA_X$-module we denote by $\cM^{loc}$ the $\cA_X^{loc}$-module $\C((\hbar)) \otimes \cM$.

\begin{corollary}
Let $X$ be a compact complex manifold endowed with a DQ-algebroid $\cA_X$ and let $\cK \in \Der^b_{\coh}(\cA_{X \times X^a})$. Then,
\begin{equation*}
\sum_i (-1)^i \tr(\Hn^i(\Phi_\cK))=\int_X \delta^\ast \ch(\gr_\hbar \cK) \cup \td_X(TX)
\end{equation*}
where $\ch(\gr_\hbar \cK)$ is the Chern class of $\gr_\hbar \cK$, $\td_X(TX)$ is the Todd class of the tangent bundle $TX$ and $\delta^\ast$ is the pullback by the diagonal embedding.
\end{corollary}

\begin{proof}
By Corollary \ref{Lunts}, we have $\tr(\Phi_\cK)= \chi(\RHom_{\cA_X}(\w_X^{-1}, \cK))$ and 
\begin{equation*}
\chi(\RHom_{\cA_X}(\w_X^{-1}, \cK))=\chi(\RHom_{\cA^{loc}_X}((\w_X^{-1})^{loc}, \cK^{loc})). 
\end{equation*}
By Corollary 5.3.5 of \cite{KS3}, we have
\begin{equation*}
\scalebox{0.94}{$
\chi(\RHom_{\cA^{loc}_X}((\w_X^{-1})^{loc}, \cK^{loc}))=\int_{X \times X} \ch(\delta_\ast \mathcal{O}_X) \cup \ch(\gr_\hbar \cK) \cup \td_{X\times X}(T(X \times X)).$}
\end{equation*}
Applying the Grothendieck-Riemann-Roch theorem, we have
\begin{align*}
\sum_i (-1)^i \tr(\Hn^i(\Phi_\cK))&=\int_X \ch(\gr_\hbar \cK) \cup \delta_\ast \td_{X}(TX)\\
&=\int_X \delta^\ast \ch(\gr_\hbar \cK) \cup \td_{X}(TX).
\end{align*}%
\end{proof}

We denote by $d_X$ the complex dimension of $X$. In the symplectic case, we have according to \cite[§6.3]{KS3} 
\begin{theorem}
If $X$ is a complex symplectic manifold, the complex $\Hoc(\cA_X^{loc})$ is concentrated in degree $-d_X$ and there is a canonical isomorphism
\begin{equation*}
\tau_X :\Hoc(\cA_X^{loc}) \stackrel{\sim}{\underset{\tau_X}{\to}} \C_X^{\hbar, loc} \lbrack d_X \rbrack.
\end{equation*} 
\end{theorem}

We refer the reader to section 6.2 and 6.3 of \cite{KS3} for a precise description of $\tau_X$. According to \cite[Definition 6.3.2]{KS3}, the Euler class of a $\cA_X^{loc}$-module is defined by
\begin{definition}
Let $\cM \in \Der_{\coh}^b(\cA_X^{loc})$. We set
\begin{equation*}
\eu(\cM)=\tau_X(\hh_X(\cM)) \in \Hn^{d_X}_{\Supp(\cM)}(X;\C_X)
\end{equation*}
and call $\eu_X(\cM)$ the Euler class of $\cM$.
\end{definition}

Therefore, we have the following

\begin{proposition}
Let $X$ be a compact complex symplectic manifold and let $\cK \in \Der^b_{\coh}(\cA_{X \times X^a})$. Then,
\begin{equation*}
\sum_i (-1)^i \tr(\Hn^i(\Phi_\cK))=\int_{X \times X} \eu(\cC_X^{loc}) \cup \eu(\cK^{loc})
\end{equation*}
where $\cup$ is the cup product.
\end{proposition}

\begin{proof}
It is a direct consequence of \cite[§6.3]{KS3} and of Theorem \ref{preLunts}. 
\end{proof}

\begin{remark}
Similarly, it is possible to apply the results of Section \ref{generalframe} to the case of dg algebras to recover the Lefschetz-Lunts formula for dg modules.
\end{remark}

%

\end{document}